\documentclass{amsart}
\usepackage[dvipdfmx]{graphicx}
\usepackage{amssymb,amsmath,amsthm,amscd,amsfonts}

\newtheorem{theorem}{Theorem}[section]
\newtheorem{proposition}[theorem]{Proposition}

\theoremstyle{definition}
\newtheorem{definition}[theorem]{Definition}

\newtheorem{remark}[theorem]{Remark}

\newtheorem*{RPL*}{Relation to previous literature}
\newtheorem*{OTP*}{Organization of this paper}
\newtheorem*{Ack*}{Acknowledgements}

\begin{document}


\title[]
{Lagrangian self-similar solutions in gradient shrinking K\"ahler-Ricci solitons}


\author{Hikaru Yamamoto}
\address{Department of Mathematics, Faculty of Science, Tokyo University of Science, 1-3 Kagurazaka, Shinjuku-ku, Tokyo 162-8601, Japan}
\email{hyamamoto@rs.tus.ac.jp}


\begin{abstract} 
In this paper, we give a lower bound estimate for the diameter of a Lagrangian self-shrinker in a gradient shrinking K\"ahler-Ricci soliton 
as an analog of a result of A. Futaki, H. Li and X.-D. Li \cite{FutakiLiLi} for a self-shrinker in a Euclidean space. 
We also prove an analog of a result of H.-D. Cao and H. Li \cite{CaoLi} about the non-existence of compact self-expanders in a Euclidean space. 
\end{abstract} 






\thanks{This work was supported by Grant-in-Aid for JSPS Fellows Grant Number 13J06407 and the Program for Leading Graduate Schools, MEXT, Japan. }


\maketitle

\section{Introduction}
 
A gradient shrinking K\"ahler-Ricci soliton is a K\"ahler manifold $(N,\omega,g,J)$ with a smooth function $f:N\to\mathbb{R}$ satisfying 
\begin{align}\label{grasol}
\mathrm{Ric}(g)+\mathrm{Hess}_{g}f=g. 
\end{align}
By the equation (\ref{grasol}), it follows that the $(2,0)$-part of $\mathop{\mathrm{Hess}}f$ is zero. 
Hence it is clear that the $(1,0)$-part of $\nabla f$ is a holomorphic vector field on $N$. 
By a simple calculation, it is proved that the gradient of $R(g)+|\nabla f|^2-2 f$ is zero, 
and we put a constant $C_{0}$ by 
\begin{align}\label{constC0}
C_{0}:=R(g)+|\nabla f|^2-2 f, 
\end{align}
where $R(g)$ is the scalar curvature of $(N,g)$. 
It is proved that $R(g)\geq 0$ for a complete gradient shrinking Ricci soliton by an application of Corollary 2.5 in \cite{Chen}. 

For an immersion $F\colon L\to N$, we get a section $(\nabla f)\circ F \in \Gamma(L,F^{*}(TN))$, 
and we usually omit the symbol $\circ F$, for short. 
\begin{definition}\label{defofselsimsol}
An immersion map $F\colon L\to N$ is called a self-similar solution if it satisfies 
\begin{align}\label{selsimsol}
H=\lambda{\nabla f}^{\bot}
\end{align}
for some constant $\lambda \in \mathbb{R}$, 
where $H$ is the mean curvature vector field of $F$ and $\bot$ denotes the projection onto the normal bundle of $L$. 
It is called a self-shrinker, a steady soliton or a self-expander when $\lambda<0$, $\lambda=0$ or $\lambda>0$, respectively. 
\end{definition}

For example, a function $f(z^1,\dots,z^m):=\frac{1}{2}(|z^1|^2+\dots+|z^m|^2)$ on $\mathbb{C}^m$ with the standard K\"ahler structure satisfies the identity (\ref{grasol}), 
and it satisfies $\nabla f(x)=x$ under a natural identification of points and tangent vectors for all points $x\in \mathbb{C}^{m}\cong\mathbb{R}^{2m}$. 
Hence the equation (\ref{selsimsol}) coincides with $H_{x}=\lambda x^{\bot}$ for all points $x\in F(L)\subset \mathbb{C}^m\cong \mathbb{R}^{2m}$, 
and Definition~\ref{defofselsimsol} can be considered as a generalization of a self-similar solution in a Euclidean space to in a gradient shrinking Ricci soliton. 

There are many results about self-similar solutions in a Euclidean space. 
By a generalization of the notion of a self-similar solution in a Euclidean space to in a gradient shrinking Ricci soliton as in Definition \ref{defofselsimsol}, 
we can discuss which results about self-similar solutions in a Euclidean space also hold in a gradient shrinking Ricci soliton. 
As an example of such results, it is proved that a part of a result due to Smoczyk also holds in a gradient shrinking K\"ahler-Ricci soliton. 
More precisely, in the proof of Theorem 2.3.5 in \cite{Smoczyk2}, Smoczyk proved that 
every compact Lagrangian self-similar solution with exact mean curvature form is a minimal submanifold in $\mathbb{C}^n$, 
and as a generalization of this statement, 
it is proved  in \cite{Yamamoto2} that every compact Lagrangian self-similar solution with exact mean curvature form is a minimal submanifold in a gradient shrinking K\"ahler-Ricci soliton. 

In this paper, we give further two results which are already established when $(N,g)$ is a Euclidean space. 
The first result is an analog of Theorem 4.3 of A. Futaki, H. Li and X.-D. Li \cite{FutakiLiLi} under the Lagrangian assumption. 
This gives a lower bound of the diameter of a Lagrangian self-shrinker in a gradient shrinking K\"ahler-Ricci soliton. 

\begin{theorem}\label{main1}
Let $(N,\omega,g,J)$ be a $2m$-dimensional gradient shrinking K\"ahler-Ricci soliton with potential function $f\colon N\to\mathbb{R}$ satisfying the equation (\ref{grasol}). 
Let $F\colon L\to N$ be a compact Lagrangian self-shrinker with 
\[H=-\frac{1}{2}{\nabla f}^{\bot}. \]
Assume that $F(L)$ is not contained in $\{\, f=m-\frac{C_{0}}{2} \,\}$, where $C_{0}$ is a constant defined by (\ref{constC0}). 
Then we have 
\[\mathrm{diam}(L,F^{*}g)\geq \frac{\pi}{\sqrt{\frac{3}{4}+\frac{m}{2}(K_{0}+A_{0}^2)}}, \]
for constants $K_{0},A_{0}\geq 0$ satisfying $|K_{N}|\leq K_{0}$ and $|A|\leq A_{0}$, 
where $K_{N}$ is the sectional curvature of $(N,g)$ and $A$ is the second fundamental form of $F$. 
\end{theorem}

The second result is an analog of Proposition 5.3 of H.-D. Cao and H. Li \cite{CaoLi} under the Lagrangian assumption. 
This theorem states the non-existence of compact Lagrangian self-expanders in a certain gradient shrinking K\"ahler-Ricci soliton. 

\begin{theorem}\label{main2}
Let $(N,\omega,g,J)$ be a $2m$-dimensional gradient shrinking K\"ahler-Ricci soliton with potential function $f\colon N\to\mathbb{R}$ satisfying the equation (\ref{grasol}) and assume that $R(g)<2m$, 
then there exists no compact Lagrangian self-expander in $N$. 
\end{theorem}

The rest of this paper is organized as follows. 
In Section~\ref{Chara}, we give some characterization of self-similar solutions in gradient shrinking Ricci solitons. 
In Section~\ref{aaa}, we give a proof of Theorem \ref{main1} and \ref{main2}. 

\begin{Ack*}
I would like to thank my supervisor, A. Futaki for many useful discussions and constant encouragement. 
\end{Ack*}
\section{Characterization of self-similar solutions}\label{Chara}
In this section, we give some characterization of self-similar solutions in gradient shrinking Ricci solitons and review a result in \cite{Yamamoto2}. 

The first characterization is given by the variation of the weighted volume as follows. 
Let $(N,g,f)$ be a $n$-dimensional gradient shrinking Ricci soliton with potential function $f$ satisfying (\ref{grasol}). 
For an $m$-dimensional compact manifold $L$ and a constant $\lambda\in\mathbb{R}$, we define the weighted volume functional $\mathcal{F}_{\lambda}$ by 
\begin{align*}
\mathcal{F}_{\lambda}(F):=\int_{L}e^{\lambda f}d\mu(F^{*}g)
\end{align*}
for each immersion $F\colon L\to N$, 
where $d\mu(F^{*}g)$ is the induced measure on $L$ with respect to the metric $F^{*}g$. 

\begin{proposition}
Let $F\colon L\to N$ be an immersion and $\lambda\in\mathbb{R}$ be a constant. 
Then the following three conditions are equivalent. 
\begin{enumerate}
\item $F$ is a self-similar solution with $H=\lambda{\nabla f}^{\bot}$. 
\item $F$ is a minimal immersion with respect to a metric $e^{2\lambda f/m}g$ on $N$. 
\item $F$ is a critical point of $\mathcal{F}_{\lambda}$. 
\end{enumerate}
\end{proposition}
The equivalence of (1) and (2) is proved in \cite{Yamamoto2}, 
and the equivalence of (2) and (3) can be easily proved by the equality 
\[\int_{L}e^{\lambda f}d\mu(F^{*}g)=\int_{L}d\mu(F^{*}(e^{2\lambda f/m}g)). \]
The equivalence of (1) and (3) can be considered as a generalization of Proposition 3.6 in \cite{ColdingMinicozzi}. 

The second characterization is given by the asymptotic behavior of a Ricci-mean curvature flow, the coupled equation of the Ricci flow and the mean curvature flow. 
Let $(N,\omega,g,J)$ be a compact $2m$-dimensional complete gradient shrinking K\"ahler-Ricci soliton with a potential function $f\colon N\to \mathbb{R}$ satisfying (\ref{grasol}). 
Fix a time $T>0$. Then for $t\in[0,T)$ we define $g_{t}:=(T-t)\Phi_{t}^{*}g$, 
where $\{\Phi_{t}\colon N\to N\}_{t\in(-\infty,T)}$ is the 1-parameter family of holomorphic automorphisms of $(N,J)$ with $\Phi_{0}=\mathrm{id}_{N}$ 
generated by the time dependent vector field $\frac{1}{2(T-t)}\nabla f$. 
Then $g_{t}$ is a solution of K\"ahler-Ricci flow, that is, the associated K\"ahler form $\omega_{t}(\cdot,\cdot):=g_{t}(J\cdot,\cdot)$ satisfies 
\begin{align*}
\frac{\partial}{\partial t}\omega_{t}=-\rho(\omega_{t}), 
\end{align*}
where $\rho(\cdot,\cdot):=\mathrm{Ric}(J\cdot,\cdot)$ is the Ricci form of $\omega_{t}$. 
Here we review the main result in \cite{Yamamoto2}. 
Let $L$ be an $m$-dimensional compact manifold and $F\colon L\times[0,T)\to N$ be a solution of Ricci-mean curvature flow along $g_{t}=(T-t)\Phi_{t}^{*}g$, 
that is, $F$ satisfies 
\begin{align*}
\frac{\partial}{\partial t}F_{t}=H(F_{t}), 
\end{align*}
where $H(F_{t})$ is the mean curvature vector field of $F_{t}(\cdot):=F(\cdot,t)$ calculated by $g_{t}$ at each time $t$. 
Assume that the initial immersion $F_{0}\colon L\to N$ is a Lagrangian immersion for the initial K\"ahler form $\omega$.  
Then, it follows that $F_{t}\colon L\to N$ is also a Lagrangian immersion with respect to $\omega_{t}$ for all $t\in[0,T)$ (c.f. \cite{LotayPacini}). 
That is, the Lagrangian condition is preserved under a Ricci-mean curvature flow along a K\"ahler-Ricci flow. 
We further assume that $F$ develops a singularity of type I, that is, 
the norm of the second fundamental form of $F_{t}$ (denoted by $A(F_{t})$) satisfies 
\[\limsup_{t\to T}\bigl(\sqrt{T-t}\max_{L}|A(F_{t})|\bigr)<\infty. \]

Then, in \cite{Yamamoto2}, it is proved that for any sequence $t_{j}\to T$ and any point $p_{0}\in L$ 
the family of pointed immersions $\tilde{F}_{j}\colon (L,p_{0}) \to N$ defined by $\tilde{F}_{j}:=\Phi_{t_{j}}\circ F_{t_{j}}$
subconverges to a pointed immersion $\tilde{F}_{\infty}\colon (L_{\infty},p_{\infty})\to N$ satisfying 
\[H(\tilde{F}_{\infty})=-\frac{1}{2}{\nabla f}^{\bot}. \]
This result can be considered as a generalization of Huisken's result in \cite{Huisken} for a mean curvature flow in a Euclidean space. 

Since each $\tilde{F}_{j}$ is a Lagrangian immersion in $(N,\omega)$ and the Lagrangian condition ($\tilde{F}_{j}^{*}\omega=0$) is a closed condition, 
it follows that $\tilde{F}_{\infty}\colon L_{\infty}\to N$ is a Lagrangian immersion, that is, the Lagrangian self-shrinker. 
Hence a Lagrangian self-shrinker is an asymptotic model of a Lagrangian mean curvature flow with a type I singularity 
along a K\"ahler-Ricci flow constructed from a gradient shrinking K\"ahler-Ricci soliton.  

\begin{remark}
Actually, the same statement also holds under some additional assumptions even though $N$ is non-compact and complete, see \cite{Yamamoto2} for detail. 
The differences of factor $2$ or $1/2$ between coefficients appeared in some formula in this paper and those in \cite{Yamamoto2} 
arise from the difference of factor $2$ between the Ricci flow equation $\partial_{t}g_{t}=-2\mathrm{Ric}(g_{t})$ and the K\"ahler-Ricci flow equation $\partial_{t}\omega_{t}=-\rho(\omega_{t})$. 
\end{remark}
\section{Proofs of Theorem \ref{main1} and \ref{main2}}\label{aaa}
First, we give a proof of Theorem \ref{main1}. 
The proof is an analog of the proof of Theorem 4.3 of \cite{FutakiLiLi}. 
As the first step, we prove that the weighted Laplacian $\underline{\Delta}_{\phi}$ defined below has an eigenvalue 1. 
In the second step, we use Theorem 1.1 in \cite{FutakiLiLi} giving an estimate for the first eigenvalue of the weighted Laplacian. 

\begin{proof}[Proof of Theorem \ref{main1}]
Put a smooth function on $L$ by $\phi:=\frac{1}{2}f\circ F$ and consider the weighted Laplacian 
\[\mathcal{L}:=\underline{\Delta}_{\phi}:=\underline{\Delta}-\overline{\nabla}\phi\cdot\overline{\nabla}\]
which acts on $C^{\infty}(L)$, where $\underline{\Delta}$ and $\overline{\nabla}$ is the Laplacian and the gradient on $(L,F^{*}g)$. 
Since 
\begin{align}\label{4}
\underline{\Delta}(f\circ F)=\mathrm{tr}^{\top}\mathrm{Hess}_{g}f + g(\nabla f, H), 
\end{align}
we have 
\begin{align*}
\mathcal{L}\phi=&\frac{1}{2}\underline{\Delta}(f\circ F)-\frac{1}{4}(F^{*}g)(\overline{\nabla}(f\circ F),\overline{\nabla}(f\circ F))\\
=&\frac{1}{2}\left(\mathrm{tr}^{\top}\mathrm{Hess}_{g}f + g(\nabla f, H)\right)-\frac{1}{4}|{\nabla f}^{\top}|^2, 
\end{align*}
where $\mathrm{tr}^{\top}$ is the tangential trace, that is, $\mathrm{tr}^{\top}B:=\mathrm{tr}_{F^{*}g}(F^{*}B)$ for a 2-tensor $B$ on $N$. 
Since $H=-\frac{1}{2}{\nabla f}^{\bot}$, we have 
\[\frac{1}{2}g(\nabla f, H)-\frac{1}{4}|{\nabla f}^{\top}|^2=-\frac{1}{4}|{\nabla f}^{\bot}|^2-\frac{1}{4}|{\nabla f}^{\top}|^2=-\frac{1}{4}|\nabla f|^2. \]
Since $(N,g)$ is a gradient shrinking Ricci soliton satisfying (\ref{grasol}), we have 
\begin{align}\label{5}
\mathrm{tr}^{\top}\mathrm{Hess}_{g}f=\mathrm{tr}^{\top}(g-\mathrm{Ric}(g))=m-\mathrm{tr}^{\top}\mathrm{Ric}(g). 
\end{align}
Furthermore, since $F:L\to N$ is a Lagrangian immersion and $(N,g,J)$ is a K\"ahler manifold, we have 
\begin{align}\label{6}
\begin{aligned}
\mathrm{tr}^{\top}\mathrm{Ric}(g)=&\sum_{i=1}^{m}\mathrm{Ric}(g)(F_{*}e_{i},F_{*}e_{i})\\
=&\frac{1}{2}\sum_{i=1}^{m}\mathrm{Ric}(g)(F_{*}e_{i},F_{*}e_{i})+\frac{1}{2}\sum_{i=1}^{m}\mathrm{Ric}(g)(JF_{*}e_{i},JF_{*}e_{i})=\frac{1}{2}R(g), 
\end{aligned}
\end{align}
for an orthonormal basis $e_{1},\dots,e_{m}$ on $(L,F^{*}g)$. Hence we have 
\begin{align*}
\mathcal{L}\phi=\frac{m}{2}-\frac{1}{4}(R(g)+|\nabla f|^2)
=\frac{m}{2}-\frac{1}{4}(C_{0}+2f)
=\left( \frac{m}{2}-\frac{C_{0}}{4} \right) -\phi. 
\end{align*}
Since $\mathcal{L}(\mathrm{const})=0$, we have 
\[\mathcal{L}\left(\left( \frac{m}{2}-\frac{C_{0}}{4} \right) -\phi \right)=-\left(\left( \frac{m}{2}-\frac{C_{0}}{4} \right) -\phi \right). \]
By the assumption, $\frac{m}{2}-\frac{C_{0}}{4}-\phi\neq 0$. 
Hence we have proved that 1 is an eigenvalue of the weighted Laplacian $\mathcal{L}=\Delta_{\phi}$. 

Next, we prove that 
\begin{align}\label{keybd}
\mathrm{Ric}(F^{*}g)+\mathrm{Hess}_{F^{*}g}\phi  \geq \kappa F^{*}g, 
\end{align}
for a Riemannian manifold $(L,F^{*}g)$, where 
\[\kappa:=\frac{1}{2}-m(K_{0}+A_{0}^2). \]
Let $X$ be a tangent vector on $L$ and $e_{1},\dots,e_{m}$ be an orthonormal basis on $(L,F^{*}g)$. 
Then, by the Gauss equation, we have 
\begin{align*}
\mathrm{Ric}(F^{*}g)(X,X)=&\sum_{i=1}^{m}\mathrm{Rm}(F^{*}g)(X,e_{i},X,e_{i})\\
=&\sum_{i=1}^{m}\mathrm{Rm}(g)(F_{*}X,F_{*}e_{i},F_{*}X,F_{*}e_{i})\\
&-\sum_{i=1}^{m}|A(X,e_{i})|^2+g(A(X,X),H). 
\end{align*}
Furthermore, we have 
\begin{align*}
\mathrm{Hess}_{F^{*}g}\phi(X,X)=&\frac{1}{2}\mathrm{Hess}_{F^{*}g}(f\circ F)(X,X)\\
=&\frac{1}{2}\biggl(\mathrm{Hess}_{g}f(F_{*}X,F_{*}X)+g(A(X,X),\nabla f)\biggr)\\
=&\frac{1}{2}|X|^2-\frac{1}{2}\mathrm{Ric}(g)(F_{*}X,F_{*}X)-g(A(X,X),H), 
\end{align*}
where we used $\mathrm{Hess}_{g}f=g-\mathrm{Ric}$ and ${\nabla f}^{\bot}=-2H$. 
Since 
\begin{align*}
&\sum_{i=1}^{m}\mathrm{Rm}(g)(F_{*}X,F_{*}e_{i},F_{*}X,F_{*}e_{i})-\frac{1}{2}\mathrm{Ric}(g)(F_{*}X,F_{*}X)\\
=&\frac{1}{2}\sum_{i=1}^{m}\mathrm{Rm}(g)(F_{*}X,F_{*}e_{i},F_{*}X,F_{*}e_{i})-\frac{1}{2}\sum_{i=1}^{m}\mathrm{Rm}(g)(F_{*}X,JF_{*}e_{i},F_{*}X,JF_{*}e_{i})\\
\geq& -\frac{1}{2}mK_{0}|X|^2-\frac{1}{2}mK_{0}|X|^2=-mK_{0}|X|^2
\end{align*}
and 
\[-\sum_{i=1}^{m}|A(X,e_{i})|^2\geq -mA_{0}|X|^2, \]
we have 
\[\mathrm{Ric}(F^{*}g)(X,X)+\mathrm{Hess}_{F^{*}g}\phi(X,X)  \geq \left(\frac{1}{2}-m(K_{0}+A_{0}^2)\right)|X|^2. \]
Hence the inequality (\ref{keybd}) holds. 

Thus, by Theorem 1.1 in \cite{FutakiLiLi}, we have 
\begin{align}\label{FLL}
1\geq \sup_{s\in(0,1)}\left\{ 4s(1-s)\frac{\pi^2}{d^2} + s\kappa \right\}, 
\end{align}
where $d:=\mathrm{diam}(L,F_{*}g)$. 
Choosing $s=\frac{1}{2}$ in (\ref{FLL}), we have 
\begin{align*}
d\geq \frac{\pi}{\sqrt{1-\frac{1}{2}\kappa}}=\frac{\pi}{\sqrt{\frac{3}{4}+\frac{m}{2}(K_{0}+A_{0}^2)}}. 
\end{align*}
Hence the proof is completed. 
\end{proof}

Next, we prove Theorem \ref{main2}. 
The proof is an analog of the proof of Proposition 5.3 of \cite{CaoLi}. 
\begin{proof}[Proof of Theorem \ref{main2}]
Assume that there exists a compact Lagrangian self-expander $F\colon L^{m} \to N^{2m}$ ({\it compact} means $M$ is compact) with $H=\lambda{\nabla f}^{\bot}$ for some positive constant $\lambda > 0$. 
Then, using computations (\ref{4}), (\ref{5}) and (\ref{6}) in the proof of Theorem \ref{main1} and the equality $H=\lambda{\nabla f}^{\bot}$, we have 
\begin{align*}
\underline{\Delta}(f\circ F)&=\mathrm{tr}^{\top}\mathrm{Hess}_{g}f + g(\nabla f, H)\\
&=m-\frac{1}{2}R(g)+\frac{1}{\lambda}|H|^2, 
\end{align*}
where $\underline{\Delta}$ is the Laplacian on $(L,F^{*}g)$. Hence, by the divergence theorem, we have 
\begin{align*}
0=\int_{L}\underline{\Delta}(f\circ F)d\mu(F^{*}g)=\int_{L}\left( m-\frac{1}{2}R(g)+\frac{1}{\lambda}|H|^2  \right)d\mu(F^{*}g). 
\end{align*}
However, the right hand side is strictly positive by the assumption $R(g) <  2m$. 
This leads to a contradiction. 
\end{proof}


\begin{thebibliography}{9}

\bibitem{CaoLi}
H.-D. Cao and H. Li. 
\newblock A gap theorem for self-shrinkers of the mean curvature flow in arbitrary codimension. 
\newblock {\it Calc. Var. Partial Differential Equations} 46 (2013), no. 3-4, 879--889. 

\bibitem{Chen}
B.-L. Chen. 
\newblock Strong uniqueness of the Ricci flow. 
\newblock {\em J. Differential Geom.}, 82 (2009), no. 2, 363--382.

\bibitem{ColdingMinicozzi}
T. H. Colding and W. P. Minicozzi II. 
\newblock Generic mean curvature flow I: generic singularities. 
\newblock  {\em Ann. of Math.} (2) 175 (2012), no. 2, 755--833.

\bibitem{FutakiLiLi}
A. Futaki, H. Li and X.-D. Li. 
\newblock On the first eigenvalue of the Witten-Laplacian and the diameter of compact shrinking solitons.
\newblock {\em Ann. Global Anal. Geom.} 44 (2013), no. 2, 105--114.

\bibitem{Huisken}
G. Huisken.
\newblock Asymptotic behavior for singularities of the mean curvature flow.
\newblock {\em J. Differential Geom.}, 31(1):285--299, 1990.

\bibitem{LotayPacini}
J. D. Lotay and T. Pacini.
\newblock Coupled flows, convexity and calibrations: Lagrangian and totally real geometry.
\newblock arXiv:1404.4227.

\bibitem{Smoczyk2}
K. Smoczyk. 
\newblock The Lagrangian mean curvature flow. 
\newblock Univ. Leipzig (Habil.-Schr.), 2000.

\bibitem{Yamamoto2}
H. Yamamoto.
\newblock Ricci-mean curvature flow in gradient shrinking Ricci solitons. 
\newblock arXive:1501.06256.

\end{thebibliography}
\end{document}